\documentclass[11pt, reqno]{amsart}

\usepackage{amsmath}
\usepackage{amssymb}
\usepackage{amsfonts}
\usepackage{graphicx}
\usepackage{amsthm}
\usepackage{enumerate}
\usepackage{lscape}
\usepackage{dsfont}
\usepackage{color}
\usepackage{mathtools}

\usepackage{setspace}
\newcommand{\ol}{\mathrm{Hol}}
\newcommand{\Ric}{\mathrm{Ric}}

\newcommand{\C}{\mathds{C}}            
\newcommand{\de}{\partial}

\newtheorem{theor}{Theorem}

\newtheorem{defin}[theor]{Definition}
\newtheorem{lem}[theor]{Lemma}

\newtheorem{remark}[theor]{Remark}
\newtheorem{conj}[theor]{Conjecture}

\begin{document}

\title[On the  Szeg\"o kernel of  Cartan--Hartogs domains]{On the  Szeg\"o kernel of  Cartan--Hartogs domains}
\author[A. Loi, D. Uccheddu, M. Zedda]{Andrea Loi, Daria Uccheddu, Michela Zedda}
\address{Dipartimento di Matematica e Informatica, Universit\`{a} di Cagliari,
Via Ospedale 72, 09124 Cagliari, Italy}
\email{loi@unica.it; daria.uccheddu@tiscali.it; michela.zedda@gmail.com  }
\date{}

\thanks{
The first author was supported  by Prin 2010/11 -- Variet\`a reali e complesse: geometria, topologia e analisi armonica -- Italy;
the third author was supported by the project FIRB ``Geometria Differenziale e teoria geometrica delle funzioni''. All the authors were supported by  INdAM-GNSAGA - Gruppo Nazionale per le Strutture Algebriche, Geometriche e le loro Applicazioni.
}
\date{}

\begin{abstract}
Inspired by the work of Z. Lu and G. Tian \cite{LuTian} in the compact setting, in this paper we address the problem of studying  the Szeg\"o kernel of the disk bundle
over a noncompact  K\"ahler manifold.
In particular we compute the Szeg\"o kernel of the disk bundle over a  
Cartan-Hartogs domain based on a bounded symmetric domain.
The main ingredients in our analysis  are the fact that every Cartan-Hartogs domain can be viewed as an \lq\lq iterated'' disk bundle
over its base and the the ideas given in \cite{LoiArezzo} for the computation of the Szeg\"o kernel of the disk bundle over an Herimitian symmetric space of compact type.

\end{abstract}

\maketitle

\section{Introduction}
Let $(L,h)$ be an Hermitian line bundle over a K\"ahler manifold $(M,\omega)$ of complex dimension $n$ such that $\Ric(h)=\omega$, where $\Ric(h)$ is a two-form on $M$ whose local expression is given by:
\begin{equation}\label{riccih}
\Ric(h)=-\frac{i}{2}\partial \bar \partial \log h(\sigma(x),\sigma(x)),
\end{equation}
for a trivializing holomorphic section $\sigma: U\rightarrow L\setminus \{0\}$. 
 In the (pre)quantum mechanics terminology the pair $(L, h)$ is a {\em geometric quantization} of $(M, \omega)$ and $L$ is called the {\em quantum line bundle}.
For all integers $m>0$ consider the line bundle $(L^{\otimes m},h_m)$ over $(M,\omega)$ with $\Ric(h_m)=m\omega$.
Let  $\mathcal{H}_m$ be the complex Hilbert space consisting of  the $L^{\otimes m}$'s global holomorphic sections bounded with respect to the norm generated by the $L^2$-product:
$$\left<s,t \right>_m=\int_M h_m(s(x),t(x))\frac{\omega^n}{n!}(x),$$ for $s,t\in \mathcal{H}_m $. Note that if $M$ is compact, then $\mathcal{H}_m=H^0(L^{\otimes  m})$ is  finite dimensional.
Given an orthonormal basis $s^m=(s_0^m,\dots, s_{N_m}^m)$ of $\mathcal{H}_m$  (with $N_m+1=\dim \mathcal{H}_m\leq \infty$) with respect to $\left<\cdot,\cdot \right>_m$, one can define a smooth and positive real valued function on $M$, called the {\em Kempf's distortion function}: 
\begin{equation}\label{epsilon}
T_{m}(x):=\sum_{j=0}^{N_m} h_m(s_j^m(x),s_j^m(x)).
\end{equation}
As suggested by the notation, it is not difficult to verify that this function depends only on the K\"ahler form $m\, \omega$ and not on the orthonormal basis chosen.
When $M$ is compact, G. Tian \cite{ti0} and W. Ruan
\cite{ru} solved a conjecture posed by Yau by proving that the metric $g$,
associated to the form $\omega$, is the $C^{\infty}$-limit
of Bergman metrics. Zelditch
\cite{ze} generalized Tian--Ruan's theorem by proving the existence of a complete
asymptotic expansion in the $C^\infty$ category, namely
\begin{equation}\label{asymptoticZ}
T_{m}(x) \sim \sum_{j=0}^\infty  a_j(x)m^{n-j},
\end{equation}
where  $a_j$, $j=0,1, \ldots$, are smooth coefficients with $a_0(x)=1$, and
for any nonnegative integers $r,k$ the following estimate holds:
\begin{equation}\label{rest}
\left |\left |T_{m}(x)-
\sum_{j=0}^{k}a_j(x)m^{n-j}\right |\right |_{C^r}\leq C_{k, r}m^{n-k-1},
\end{equation}
where $C_{k, r}$ is a constant depending on $k$, $r$ and on the
K\"{a}hler form $\omega$, and $||\cdot\nobreak||_{C^r}$ denotes  the $C^r$
norm in local coordinates (notice that similar asymptotic expansions
were used in  \cite{cgr1}, \cite{cgr2}, \cite{cgr3}, \cite{cgr4},
\cite{mo1} and \cite{mo2} to construct  star products on
K\"{a}hler manifolds).

Later on, Lu \cite{Lu}, by means of  Tian's peak section method,
proved  that each of the coefficients $a_j(x)$ in
(\ref{asymptoticZ}) is a polynomial of the curvature and its
covariant derivatives at $x$ of the metric $g$. Such polynomials can
be found by finitely many algebraic operations. Furthermore, Lu computes the 
first three coefficients $a_1$, $a_2$ and $a_3$ of this expansion
 (see also \cite{loianal} and
\cite{loismooth} for the computations of the coefficients $a_j$'s
through Calabi's diastasis function). 
The expansion
(\ref{asymptoticZ}) is called the  {\em TYZ (Tian--Yau--Zelditch)
expansion} and it is a key ingredient in the investigations of balanced
metrics \cite{do} (see also  \cite{arlcomm}). 
Notice that   prescribing the values of  the  coefficients
of the  TYZ expansion gives rise to interesting elliptic PDEs
as shown by Z. Lu and G. Tian  \cite{LuTian}.
The main result obtained in \cite{LuTian} is that {\em if the log--term
of the Szeg\"{o} kernel of the unit disk bundle over $M$ vanishes then  $a_k=0$, for all $k>n$}. Recall that the disk bundle over  $M$ is the strongly pseudoconvex domain $D\subset M$ defined by 
$D=\{ v\in M|\, \rho(v)>0\}$ and we denote by $X=\de D$ its boundary. Given the separable Hilbert space
${\mathcal H}^2(X)$ consisting
of all holomorphic functions on $D$ which are continuous on $X$ and satisfy: 
$$\int_{X} |f|^2 d\nu < \infty,$$ 
where $d\nu=\alpha\wedge (d\alpha)^n$ and 
$\alpha =-i\partial\rho_{|X}=i\bar\partial\rho_{|X}$ is the  contact form on $X$ associated to the strongly pseudoconvex domain $D$ (the $1$-form $\alpha$  is defined on the smooth part of $X$), the Szeg\"o kernel of $D$ is defined by: 
$${\mathcal S}(v)=\sum_{j=1}^{+\infty} f_j(v)\overline {f_j(v)},\  v\in D,$$
where $\{f_j\}_{j=1, \dots}$ is an orthonormal basis of ${\mathcal H}^2(X)$.
 A direct computation of the Szeg\"o kernel could be in general very complicated. Although, when $D\subset M$ is a strongly pseudoconvex domain with smooth boundary, the following celebrated formula due to Fefferman
(see \cite{fefferman} and also \cite{BealsFeffermanGrossman}) shows that there exist
 functions $a$ and $b$ continuous on $\bar D$ and with $a\neq0$ on $X$, such that:
\begin{equation}\label{fefferman}
 \mathcal{ S}(v)=\frac{a(v)}{\rho(v)^{n+1}}+b(v) \log \rho(v).
\end{equation}
The function $b(v)$ is called the {\em logarithmic term}  (or {\em log--term}) of the Szeg\"{o} kernel and one says that the log--term  of the Szeg\"{o} kernel of  $D$ vanishes if $b=0$.

Z. Lu has conjectured (private communication) that the converse of the above mentioned result  is true:
\begin{conj}[Lu] \label{ConjLu}
Let $(L,h)$ be a positive line bundle over a compact complex manifold $(M, \omega)$ of dimension $n$ such that $\Ric(h)=\omega$. If the coefficients $a_k$ of TYZ in (\ref{asymptoticZ}) vanish for all $k>n$, then the log--term of the Szeg\"{o} kernel of the unit disk bundle over $M$ vanishes. 
\end{conj}


In \cite{graloi} (see also \cite{LoiZeddaTaub}) the authors address the problem of the existence of a TYZ expansion in the noncompact case
and study its coefficients.

In this paper we study the analogous of the previous conjecture for an important family of noncompact 
K\"{a}hler manifolds called Cartan-Hartogs domains, defined as follows.
Let $\Omega\subset \C^d$ be a bounded symmetric domain of genus $\gamma$ and denote by $N=N(z)$ its {\em generic norm}, namely, 
\begin{equation}
N(z)=(V(\Omega) K(z, z))^{-\frac{1}{\gamma}},\nonumber
\end{equation}
where $V(\Omega)$ is the total volume of $\Omega$ with respect to the Euclidean measure of $\C^d$ and $K(z, z)$ is its Bergman kernel (see e.g. \cite{arazy} for more details). 
The {\em Cartan-Hartogs domain $M_{\Omega}^{d_0}(\mu)$  based on $\Omega$}  is  the pseudoconvex domain of $\mathbb{C}^{d+d_0}$  defined by ($\mu>0$ is a fixed constant):
\begin{equation}\label{defm}
M_{\Omega}^{d_0}(\mu)=\left\{(z,w)\in \Omega\times\C^{d_0},\ ||w||^2<N^\mu(z)\right\}.
\end{equation}
It can be equipped with the natural  K\"{a}hler form:
$$\omega_{d_0}=-\frac i2\de\bar\de \log(N^\mu (z)-||w||^2).$$
The K\"{a}hler manifold $(M_{\Omega}^{d_0}(\mu), \omega_{d_0})$ has been studied by several authors from different analytic and geometric points of view (see for example \cite{MZ} \cite{Feng}, \cite{roosformula}, \cite{yin1},  \cite{yin2} and  \cite{zedda}). 
One can consider the trivial line bundle\footnote{Due to the contractibility and pseudoconvexity of $M_{\Omega}^{d_0}(\mu)$,  {\em any} holomorphic line bundle over $M_{\Omega}^{d_0}(\mu)$ is  holomorphically trivial.} $L=M_{\Omega}^{d_0}(\mu)\times\C$ on $M_{\Omega}^{d_0}(\mu)$
endowed with the  Hermitian metric:
\begin{equation}\label{hM}
h_{d_0}(z, w; \xi)=\left(N^{\mu}(z)-||w||^2\right)|\xi|^2,\,\, \,\,(z, w)\in M_{\Omega}^{d_0}(\mu), \,\,\xi\in\C,
\end{equation}
which satisfies $\Ric(h_{d_0})=\omega_{d_0}$ (cfr. Equation \eqref{riccih}).

\vskip 0.3cm
The main result about the TYZ expansion for Cartan-Hartogs domains is expressed by the following recent result in \cite{MZ}, which shows that in this case the expansion is indeed finite, namely it is a polynomial in $m$ of degree $d+d_0=\dim M_{\Omega}^{d_0}(\mu)$ with  computable (non-constant) coefficients.
\begin{theor}[Feng-Tu]
Let $m > \max\left\{d+d_0, \frac{\gamma-1}{\mu}\right\}$, then the Kempf's distorsion function associated to $(M_{\Omega}^{d_0}(\mu), \omega_{d_0})$  can be written as: 
\begin{equation}\label{epsilonfun}
T_m(z,w)=\frac{1}{\mu^d}\sum_{k=0}^{d}\frac{ D^k\tilde X(d)}{k!}\left(1-\frac{||w||^2}{N^\mu}\right)^{d-k}\frac{\Gamma(m -d+k)}{\Gamma(m-d-d_0)},
\end{equation}
 with 
$$
 D^k\tilde X(d)=\sum_{j=0}^{k}\binom{k}{j}(-1)^j\prod_{l=1}^r\frac{\Gamma(\mu(d-j)-\gamma+2-(l+1)\frac{a}{2}+b+ra)}{\Gamma(\mu(d-j)-\gamma+1+(l-1)\frac{a}{2})}.
$$

\end{theor}
Formula (\ref{epsilonfun}) implies, in particular, that $a_k=0$ for $k>d+d_0.$
Therefore it is natural to see if Conjecture \ref{ConjLu} holds true in this (noncompact) case.\footnote{Formula (\ref{epsilonfun}) is used by Feng and Tu   to give a positive answer to a
conjecture posed by the third author of the present paper in \cite{zedda}, namely they  
prove that if coefficient $a_2$ is constant, then $M_{\Omega}^{d_0}(\mu)$ is the complex hyperbolic space.
This  formula  has been also  used in \cite{berezinCH}  to study the Berezin quantization of $(M_{\Omega}^{d_0}(\mu),\omega_{d_0})$.}

Notice  that the disk bundle of a Cartan-Hartogs $M_{\Omega}^{d_0}(\mu)$ is the Cartan-Hartogs domain $M_\Omega^{d_0+1}(\mu)$, whose the boundary of $M_\Omega^{d_0+1}(\mu)$ is not smooth being:
 $$\partial M_\Omega^1(\mu)=\partial \Omega \cup \{(z,w)\in\Omega\times\C\ | \  |w|^2=N^\mu\}.$$   \\
Thus, it does not make sense to speak of the log--term of the Szeg\"o kernel, since formula (\ref{fefferman}) applies only when the domain involved has smooth boundary. 
Nevertheless, in order to consider the case of Cartan-Hartogs domain, we give the following definition
(which in the smooth boundary case coincides with the standard one).

\begin{defin}\label{log}\rm
Let $D\subset M$ be a strongly pseudoconvex domain in a complex $n$-dimensional manifold $M$, let $X=\de D$ be its boundary with defining function $\rho>0$, i.e. $D=\{ v\in M|\, \rho(v)>0\}$. Assume that the points where $X$ fails to be smooth are of measure zero.
We say that {\em the log--term of the Szeg\"o kernel of the disk bundle vanishes} if there exists a continuous function $a$ on $\bar D$ with $a\neq 0$ on $X$, such that 
$ \mathcal{ S}(v)=\frac{a(v)}{\rho(v)^{n+1}}.$
\end{defin} 

The main result of this paper is the following:
\begin{theor}\label{mainmainresult}
The log--term  of the Szeg\"o kernel of the disk bundle over a Cartan--Hartogs domain vanishes.
\end{theor}

In the next  section   we compute the Szeg\"o kernel  of  $(M_{\Omega}^{d_0}(\mu), \omega_{d_0})$ and  prove Theorem \ref{mainmainresult}.

\section{Szeg\"{o} kernel of Cartan-Hartogs domains}\label{szegoch}

 In the following lemma, needed in the proof of Theorem \ref{mainmainresult}, we compute the volume form $\alpha\wedge (d \alpha)^d$ on the boundary $\partial M_{\Omega}^1(\mu)$ of $M_{\Omega}^1(\mu)$, namely 
 a Cartan--Hartogs domain with $d_0=1$.
 \begin{lem}
The volume form $\alpha \wedge (d \alpha)^d$ on the boundary $\partial M_{\Omega}^1(\mu)$ is given in polar coordinates $(\rho,\theta)$ by:
$$
\alpha\wedge (d\alpha)^d=\,\left(\frac{2\mu}{\gamma}\right)^dN^{\mu(d+1)-\gamma}d\theta_{w}\wedge\frac{\omega_0^d}{d!},
$$
where $\frac{\omega_0^d}{d!}$ is the standard volume form of $\mathds{C}^d$ and $\theta_w=\theta_{d+1}$.
\end{lem}
\begin{proof}
By definition $\alpha=-i\de\rho_{|\de M_{\Omega}^1(\mu)}$, where 
$\rho=N^\mu-\vert w\vert^2>0$ is the defining function of $M_{\Omega}^1(\mu)$. Thus, we get:
$$\alpha=-i\left(\sum_{j=1}^d\de_j N^\mu dz_j-\bar wdw\right).$$
Furthermore, by $d\alpha=(\de+\bar\de)$  $\alpha=-i\bar\de \de\rho$, we get: 
$$d\alpha=-i\left(\sum_{j,k=1}^dN^\mu_{j\bar k}dz_j\wedge d\bar z_k-dw\wedge d\bar w\right)=i\left(dw\wedge d\bar w-\sum_{j,k=1}^dN^\mu_{j\bar k}dz_j\wedge d\bar z_k\right),$$
$$(d\alpha)^d=i^d\left(\det(-N^\mu_{j\bar k})d\xi+\sum_{s,q=1}^d(-1)^{s+q}\det(-N^\mu_{j\bar k})_{s\bar q}d\zeta_{s\bar q}\right),$$
where we write $N^\mu_{j}=\de N^\mu/\de z_j$, $N^\mu_{\bar k}=\de N^\mu/\de\bar z_k$ and $N^\mu_{j\bar k}=\de^2 N^\mu/\de z_j\de\bar z_k$, we denote by $d\xi=dz_1\wedge d\bar z_1\wedge\dots\wedge dz_d\wedge d\bar z_d$ and by
$d\zeta_{\bar q}$, (resp. $d\zeta_{s\bar q}$) the form $d\xi$ where the term $d\bar z_q$ (resp. the terms $dz_s$, $d\bar z_q$) is replaced by  $d\bar w$ (resp. $dz_s$ with $dw$ and $dz_{\bar q}$ with $d\bar w$). Further, we write $(-N^\mu_{j\bar k})_{s\bar q}$ for the matrix $(-N^\mu_{j\bar k})$ where the $s$-th row and the $q$-th column have been deleted. 
Thus, the volume form $\alpha\wedge (d\alpha)^d$ is given by:
\begin{equation}\label{volumecontact}
\begin{split}
\alpha\wedge (d\alpha)^d=-i^{d+1}\left(\sum_{s,q=1}^d(-1)^{s+q}\right.& N^\mu_s\det(-N_{j\bar k}^\mu)_{s\bar q}dz_s\wedge d\zeta_{s\bar q}+\\
&\left.-\bar w \det(-N^\mu_{jk})dw\wedge d\xi\right).
\end{split}
\end{equation}
Observe first that:
$$
dz_s\wedge d\zeta_{s\bar q}=-dw\wedge d\zeta_{\bar q}=dw\wedge d\bar w\wedge d\xi_{\bar q},
$$
where $d\xi_{\bar q}$ is the form $d\xi$ where the term $d\bar z_{q}$ was deleted. Further, evaluating at the boundary, turning to polar coordinates $(\rho,\theta)$ and denoting $\rho_{d+1}$  by $\rho_w$ and $\theta_{d+1}$ by $\theta_w$, from $\rho_w^2=N^\mu$ one has $2\rho_wd\rho_w=\sum_{j=1}^d N_{\bar j}^\mu e^{-i\theta_j}\left(d\rho_j-i\rho_jd\theta_j\right)$ and we get:
\begin{equation}\label{wdw}
\bar wdw\wedge d\xi=\rho_w(d\rho_w+i\rho_wd\theta_w)\wedge d\xi=i N^\mu d\theta_w\wedge d\xi,
\end{equation}
and
$$
dw\wedge d\bar w=-2i\rho_w d\rho_w\wedge d\theta_w=-i\sum_{j=1}^dN^\mu_{\bar j}d\bar z_j\wedge d\theta_w,
$$
which yields
\begin{equation}\label{dz}
dz_s\wedge d\zeta_{s\bar q}=-iN^\mu_{\bar q}d\bar z_q\wedge d\theta_w\wedge d\xi_{\bar q}=-iN^\mu_{\bar q} d\theta_w\wedge d\xi.
\end{equation}
Substituting \eqref{wdw} and \eqref{dz} into \eqref{volumecontact} we get:
$$
\alpha\wedge (d\alpha)^d=\,i^{d}A\, d\theta_w\wedge d\xi=\, 2^{d}A\, d\theta_w\wedge \frac{\omega_0^d}{d!},
$$
where we used that $\frac{\omega_0^d}{d!}=\left(\frac{i}{2}\right)^d d\xi$ and we set:
$$A=N^{\mu}\det\left(\left[-N^\mu_{j\bar k}\right]\right)-\sum_{j,k=1}^d(-1)^{j+k}N^\mu_j N^\mu_{\bar k}\det\left(\left[-N^\mu_{p\bar q}\right]\right)_{j\bar k}.$$
It remains to show that:
\begin{equation}\label{A}
A=\left(\frac{\mu}{\gamma}\right)^dN^{\mu(d+1)-\gamma}.
\end{equation}

In order to prove \eqref{A}, consider the metric $g_\Omega$ of the domain $\Omega$ associated to $\omega_\Omega$ defined by $(g_{\Omega})_{j\bar k}=\frac{\de^2\log(N^\mu)}{\de z_j \de\bar z_k}$. A direct computation gives:
\begin{equation}
\begin{split}
\det(g_\Omega)=&\,\det\left(\left[\frac{ N^\mu_jN^\mu_{\bar k}-N^\mu_{j\bar k}N^\mu}{N^{2\mu}}\right]\right)\\
=&\,\frac{1}{N^{2d\mu}}\det\left(\left[N^\mu_jN^\mu_{\bar k}-N^\mu_{j\bar k}N^\mu\right]\right)\\
=&\,\frac{N^\mu_1\cdots N^\mu_d}{N^{2d\mu}}\det\left(\left[N^\mu_{\bar k}-\frac{N^\mu_{j\bar k}N^\mu}{N^\mu_j}\right]\right)\\
=&\,\frac{\prod_{h=1}^dN^\mu_hN^\mu_{\bar h}}{N^{2d\mu}}\det\left([1]+\left[-\frac{N^\mu_{j\bar k}N^\mu}{N^\mu_jN^\mu_{\bar k}}\right]\right)\\
=&\,\frac{1}{N^{d\mu}}\det\left(\left[-N^\mu_{j\bar k}\right]\right)-\frac{1}{N^{\mu(d+1)}}\sum_{j,k=1}^d(-1)^{j+k}N^\mu_j N^\mu_{\bar k}\det\left(\left[-N^\mu_{p\bar q}\right]\right)_{j\bar k}\\
=&\,\frac{A}{N^{\mu(d+1)}}.
\end{split}\nonumber
\end{equation}
Conclusion follows by:
$$\det(g_\Omega)=\left(\frac{\mu}{\gamma}\right)^d\det(g_B)=\left(\frac{\mu}{\gamma}\right)^dN^{-\gamma},$$
where $g_B=\frac\mu\gamma g_{\Omega}$ is the Bergman metric on $\Omega$ (whose determinant can be obtained easily by considering that it is K\"ahler--Einstein with Einstein constant $-2$).
\end{proof}
\begin{proof}[Proof of Theorem \ref{mainmainresult}] Observe first that by an inflation principle (see e.g. Section 2.3 in \cite{roosformula}) we can assume without loss of generality $d_0=1$. In this case the defining function $\rho (z, w)= N^\mu(z) -|w|^2$ and $$\partial M_\Omega^1(\mu)=\partial \Omega \cup \{(z,w)\in\Omega\times\C\ | \  |w|^2=N^\mu\}.$$  
Observe that although $\partial M_\Omega^1(\mu)$ is smooth only when $\Omega$ is of rank $1$ (i.e. when $\Omega$ is the complex hyperbolic space), the points where it fails to be smooth are of measure zero.
The volume form  $d\nu =\alpha\wedge (d\alpha)^d$ 
reads :
\begin{equation}\label{dmu}
d\nu= \alpha\wedge (d\alpha)^d=\left(\frac{2\mu}{\gamma}\right)^dN^{\mu(d+1)-\gamma}d\theta_w\wedge\frac{\omega_0^d}{d!},
\end{equation}
where  $\frac{\omega_0^d}{d!}$ is the standard Lebesgue measure on $\C^d$ ($\omega_0$ is the flat K\"ahler  form on 
$\C^d$).
In order to compute the Szeg\"o kernel ${\mathcal S}_{M_\Omega^1(\mu)}$ of $ M_{\Omega}^1(\mu)$ one needs to find an orthonormal basis 
of  the  separable Hilbert space
$\mathcal{H}^2(\partial M_\Omega^1(\mu))$ (Hardy space) consisting
of  all holomorphic functions $\hat s$ on $M_\Omega^1(\mu)$, continuous on $\partial M_\Omega^1(\mu) $ and such that 
$$\int_{\partial M_\Omega^1(\mu)} |\hat s|^2 d\nu < \infty .$$ 
Consider the  Hilbert space:
 $$\mathrm{H}^2_m(\Omega)=\left \{s\in \ol (\Omega) \ \left |  \ \int_{\Omega} N^{\mu m}|s(z)|^2\frac{\omega_{\Omega}^d}{d!}<\infty \right . \right \},$$
(where $\omega_{\Omega}=\frac{\gamma}{\mu}\omega_B$ is the K\"ahler form in $\Omega$ given by $\omega_\Omega=-\frac{i}{2}\de \bar \de \log N^{\mu}$)
 and the map:
\begin{equation} \label{isom}
\wedge : \mathrm{H}^2_m(\Omega)\rightarrow \mathcal{H}^2(\partial M_\Omega^1(\mu) ):\ \ s\mapsto \hat s\\	
\end{equation}
defined by
$$
\hat s(v)=2^{-\frac{d}{2}}  N(z,z)^{-\frac{\mu(d+1)}{2}}w^m s(z),\,\,\, v=(z, w)\in \partial M_\Omega^1(\mu).
$$ 
Notice that the Hardy space $\mathcal{H}^2(\partial M_\Omega^1(\mu))$  admits a Fourier decomposition into irreducible factors with respect to the natural  $S^1$-action, i.e.
$$
\mathcal{H}^2(\partial M_\Omega^1(\mu))=\bigoplus _{m=0}^{+\infty}\mathcal  H_m^2(\partial M_\Omega^1(\mu)),
$$
where 
 $\mathcal{H}_m^2(\partial M_\Omega^1(\mu)):=\{\hat s\in\mathcal{ H}^2(\partial M_\Omega^1(\mu))\:\vert\: \hat s(\lambda v)=\lambda^m \hat s(v)\}$ and $\lambda v:=(z, \lambda w)$, for $v=(z, w)$.
Since
 $$
 \frac{\omega_\Omega^d}{d!}= \left(\frac{\mu}{\gamma}\right)^dN^{-\gamma} \frac{\omega_0^d}{d!},
 $$
 it is not hard to see that the map $\wedge$ defines an isometry between 
$\mathrm{H}^2_m(\Omega)$ and $\mathcal{H}^2_m(\partial M_\Omega^1(\mu))$.
Thus, if we consider the orthogonal projection of the Szeg\"o kernel on each  $\mathcal{H}_m^2(\partial M_\Omega^1(\mu))$, we get:
\begin{equation}\label{szegoformula}
\mathcal{S}_{M_\Omega^1(\mu)}(v)=\sum_{m=0}^{+\infty}\sum_{j=0}^{+\infty} \hat s_j^m(v) \overline{\hat s_j^m(v)}
=2^{-d}N^{-\mu(d+1)}\sum_{m=0}^{+\infty}\sum_{j=0}^{+\infty} |w|^{2m}|s_j^m(z)|^2,
\end{equation}
where $s^m_j, j\ =0, 1,\dots$ is an orthonormal basis of $\mathrm{H}^2_m(\Omega)$ and $\hat s^m_j=\wedge (s^m_j)$ is the corresponding orthonormal basis for $\mathcal{H}^2_m(\partial M_\Omega^1(\mu))$.

It is well-known (for a proof, see e.g. \cite[p.77]{FauK} or \cite[Ch. XIII.1]{FauK2})
 that  $\sum_{j=0}^\infty N^{\mu m} |s_j^m(z)|^2$
 is a polynomial in $ m$ of degree $d=\dim \Omega$, hence it can be written as:
$$
\sum_{j=0}^\infty N^{\mu m} |s_j^m(z)|^2=\sum_{l=0}^{d} b_l \binom{ m+l}{l}
$$
where $b_l$ depends on the metric $g_{\Omega}$ associated to $\omega_{\Omega}$.
Thus, this formula together with (\ref{szegoformula}) yields:
\[
\begin{split}
\mathcal{S}_{M_\Omega^1(\mu)}(v)=&2^{-d}N^{-\mu(d+1)}\sum_{m=0}^\infty\sum_{l=0}^{d} |w|^{2m}N^{-\mu m } b_l \binom{ m+l}{l}\\
=&2^{-d}N^{-\mu(d+1)}\sum_{l=0}^{d} b_l \sum_{m=0}^\infty   \binom{ m+l}{l}(|w|^{2}N^{-\mu })^m  \\
=&2^{-d}N^{-\mu(d+1)}\sum_{l=0}^{d} b_l \frac{1}{\left(1-|w|^{2}N^{-\mu } \right)^{l+1}}.
\end{split} \]
That is 
\begin{equation}
\begin{split}
\mathcal{S}_{M_\Omega^1(\mu)}(v)=& 2^{-d}N^{-\mu(d+1)}\left[\frac{ b_0N^\mu}{\left(N^{\mu}-|w|^{2}\right)}+\dots +  \frac{b_d N^{\mu (d+1)}}{\left(N^{\mu}-|w|^{2}\right)^{d+1}}\right] \\
=&2^{-d}\frac{ b_0 N^{-\mu d}\left(N^{\mu}-|w|^{2}\right)^d+ \dots +b_{d-1} N^{-\mu}\left(N^{\mu}-|w|^{2}\right)^2+ b_d}{\left(N^{\mu}-|w|^{2}\right)^{d+1}}.
\end{split}\nonumber
\end{equation}
Observe that in the above expression, all terms except $b_d=d!\, m^d$ vanish once evaluated at the boundary $\partial M_\Omega^1(\mu)$.
The vanishing of the log--term of $\mathcal{S}_{M_\Omega^1(\mu)}$ (as in Definition \ref{log})   follows then  by setting:
$$
a(v)=2^{-d}\left( b_0 N^{-\mu d}\left(N^{\mu}-|w|^{2}\right)^d+ \dots +b_{d-1} N^{-\mu}\left(N^{\mu}-|w|^{2}\right)^2+ b_d\right).
$$
\end{proof}

\begin{remark}\rm
It is worth pointing  out that in \cite{Feng} it is  shown that the log-term of the   Szeg\"o kernel of $M_{\Omega}^{d_0}(\mu) \subset \mathbb{C}^{d+d_0}$ vanishes (in the sense of our  Definition \ref{log})
when the Szeg\"o kernel is  obtained using the standard volume form of $\mathbb{C}^{d+d_0}$ restricted to $\partial M_{\Omega}^{d_0}(\mu) $
instead of the volume form $d\nu=\alpha\wedge (d\alpha)^d$ used in this paper. The reader is referred also  to  \cite{roosformula} for the
proof of the vanishing of the log--term  of the Bergman kernel.
\end{remark}

\end{document}